\documentclass{amsart}

\usepackage{amssymb}
\usepackage{amsthm}
\usepackage{cases}
\usepackage{amsfonts}
\usepackage{mathrsfs}
\usepackage{amsmath}
\usepackage{extarrows}
\usepackage{color}
\newtheorem{theorem}{Theorem}[section]

\newtheorem{lemma}[theorem]{Lemma}

\newtheorem*{remark}{Remark}
\numberwithin{equation}{section}

\newcommand{\RePt}{\mathrm{Re}\,}
\newcommand{\ImPt}{\mathrm{Im}\,}

\newcommand{\calU}{\mathcal{U}}

\newcommand{\bfi}{\mathbf{i}}

\newcommand{\bfrho}{\boldsymbol{\rho}}

\begin{document}

\title[Schatten class positive Toeplitz operators]{Schatten class positive Toeplitz operators on Bergman spaces of the Siegel upper half-space}

\thanks{This work was supported by the National Natural Science Foundation of China grant 11971453.}

\author{Jiajia Si}
\email{sijiajia@mail.ustc.edu.cn}

\address{School of Science, 
Hainan University, 
Haikou, Hainan 570228, 
People’s Republic of China.}


\begin{abstract}
We characterize Schatten class membership of positive Toeplitz operators defined on the Bergman spaces over the Siegel upper half-space in terms of averaging functions and Berezin transforms in the range of $0<p<\infty$.
\end{abstract}


\maketitle

\section{Introduction}

Let $\mathbb{C}^n$ be the $n$-dimensional complex Euclidean space.
For any two points $z=(z_1,\cdots,z_n)$ and $w=(w_1,\cdots,w_n)$ in $\mathbb{C}^n$ we write
\[
z\cdot\overline{w} := z_1\overline{w}_1+\cdots+z_n\overline{w}_n,
\]
and
\[
|z|:= \sqrt{z\cdot\overline{z}} = \sqrt{|z_1|^2+\cdots+|z_n|^2}.
\]
The set
\[
\calU=\left\{ z\in\mathbb{C}^n:\ImPt z_n>|z^{\prime}|^2\right\}
\]
is the Siegel upper half-space. Here and throughout the paper, we use the notation
\[
z=(z^{\prime},z_n),\,\,\,\, \text{where}\, z^{\prime}=(z_1,\cdots,z_{n-1})\in\mathbb{C}^{n-1}\,\, \text{and}\,\, z_n\in\mathbb{C}^{1}.
\]
The Bergman space $A^2(\calU)$ is the space of all complex-valued holomorphic functions $f$ on $\calU$ such that 
\[
\int_{\calU} |f|^2 dV<\infty,
\]
where $V$ denotes the Lebesgue measure on $\calU$. It is a closed subspace of $L^2(\calU)$ and hence a Hilbert space.
The orthogonal projection from $L^2(\calU)$ onto $A^2(\calU)$ can be expressed as an integral operator:
\[
Pf(z)=\int_{\calU} K(z,w)f(w)dV(w),
\]
with the Bergman kernel
\[
K(z,w)=\frac{n!}{4\pi^n}\left[ \frac{i}{2}(\overline{w}_n-z_n)-z^{\prime} \cdot \overline{w^{\prime}} \right]^{-n-1}.
\]
See, for instance, \cite[Theorem 5.1]{Gin64}.
The operator $P$ is usually called a Bergman projection.
It is a bounded projection from $L^2(\calU)$ onto $A^2(\calU)$, see \cite[Lemma 2.8]{CR80}.

Let $\mathcal{M}_+$ be the set of all positive Borel measures $\mu$ such that
\[
\int_{\calU}\frac{d\mu(z)}{|z_n+i|^{\alpha}}<\infty
\]
for some $\alpha>0$.
Given $\mu\in \mathcal{M}_+$, the Toeplitz operator $T_{\mu}$ with symbol $\mu$ is given by
\[
T_{\mu}f(z)=\int_{\calU} K(z,w)f(w) d\mu(w)
\]
for $f\in A^2(\calU)$. In case $d\mu=gdV$, we write $T_{\mu}=T_g$. In general, $T_{\mu}$ may not even be defined on all of $A^2(\calU)$, but it is always densely defined by the fact that, for each $\alpha>n+1/2$, holomorphic functions $f$ on $\calU$
such that $f(z) = O(|z_n+i|^{-\alpha})$ form a dense subset of $A^2(\calU)$, see \cite{LS20}.

For a positive Borel measure $\mu$ on $\calU$, the Berezin transform of $\mu$ is given by
\begin{equation}\label{eqn:berezin}
\widetilde{\mu}(z) := \int_{\calU} |k_z(w)|^2 d\mu(w),\quad z\in \calU,
\end{equation}
where
\[
k_z(w) := K(z,w)/\sqrt{K(z,z)},\quad w\in\calU.
\]
For $z\in\calU$ and $r>0$, we define the averaging function
\begin{equation}\label{eqn:averaging}
\widehat{\mu}_r(z) := \frac{\mu(D(z,r))}{|D(z,r)|},
\end{equation}
where $D(z,r)$ is the  Bergman metric ball at $z$ with radius $r$ and
$|D(z,r)|$ denotes the Lebesgue measure of $D(z,r)$.

Quite recently in \cite{LS20}, we have characterized the boundedness and compactness of $T_{\mu}$ on Bergman spaces in terms of $\widetilde{\mu}$ and $\widehat{\mu}_r$. In the paper, we are going to study the membership of $T_{\mu}$ in Schatten class $S_p$ in terms of $\widetilde{\mu}$ and $\widehat{\mu}_r$. It is a natural subsequent work of \cite{LS20}. Our main result is the following.

\begin{theorem}\label{thm:main1}
Suppose $0<p<\infty$, $r>0$ and $\mu\in\mathcal{M}_+$. Then the following conditions are equivalent:
\begin{enumerate}
  \item[(i)] $T_{\mu}\in S_p$.
  \item[(ii)] $\widehat{\mu}_r \in L^p(\calU,d\lambda)$.
  \item[(iii)] $\{\widehat{\mu}_r(a_k)\}\in l^p$ for every $r$-lattice $\{a_k\}$.
  \item[(iv)] $\{\widehat{\mu}_r(a_k)\}\in l^p$ for some $r$-lattice $\{a_k\}$.  
\end{enumerate}
Moreover, if $p>n/(n+1)$, then the above conditions are also equivalent to  
\begin{enumerate}
   \item[(v)] $\widetilde{\mu}\in L^p(\calU,d\lambda)$.
\end{enumerate}
\end{theorem}

Here $d\lambda(z)=K(z,z)dV(z)$ is the M\"obius invariant measure on $\calU$ (see \cite[Proposition 1.4.12]{Kra01} for example) and the cut-off point $n/(n+1)$ is sharp. 

The equivalences of (i), (iii) and (iv) were originally proved by Luecking \cite{Lue87} for the full range $0<p<\infty$ in the case of the unit disk. Later Zhu \cite{Zhu88} extended Luecking's result to bounded symmetric domains and added condition (v), with the restricted range $1\leq p<\infty$. After that, Zhu \cite{Zhu07NY} continued to study similar characterizations for the range $0<p<1$ on the unit ball. His result reveals an interesting aspect that, the full range $0<p<1$ for characterization with averaging function is appropriate, while this is not the case for characterization with Berezin transform. 
On the setting of harmonic Bergman space of the upper half-space of $\mathbb{R}^n$, Choe et al. obtained similar characterizations, see \cite{CKY02,CKL07}. 


The paper is organized as follows. In Section 2 we collect some basic auxiliary results. Section 3 is devoted to show the equivalences of (ii)-(iv). In Section 4 we prove Theorem \ref{thm:main1} in the case of $p\geq 1$, where also we recall the notion and some elementary facts of Schatten class operators. Section 5 is devoted to show the equivalences of (i)-(iv) in the range of $0<p<1$. The characterization of Schatten class Toeplitz operators in terms of the Berezin transform is obtained in Section 6. 

Throughout the paper, the letter C will denote a positive constant that may vary at each occurrence but is independent of the essential variables. 

\section{Preliminaries}

In this section we introduce notations and collect several basic lemmas which will be used in later sections.
Throughout the paper, we write
\[
\bfrho(z,w):=\frac{i}{2}(\overline{w}_n-z_n)-z^{\prime} \cdot \overline{w^{\prime}}
\]
and let $\bfrho(z):=\bfrho(z,z)=\ImPt z_n-|z^{\prime}|^2$ for simplicity.
With this notation, the Bergman kernel of $\calU$ is 
\[
K(z,w) = \frac {n!}{4\pi^n} \frac {1}{\bfrho(z,w)^{n+1}}, \quad z,w\in \calU.
\]

For each $t>0$, we define the nonisotropic dilation $\delta_t$ by
\[
\delta_t(u)=(t u^{\prime},t^2 u_n), \quad u\in \calU.
\]
Also, to each fixed $z\in\calU$, we associate the following (holomorphic) affine
self-mapping of $\calU$:
\[
h_z(u) ~:=~ \left(u^{\prime} - z^{\prime}, u_n - \RePt z_n - 2i u^{\prime} \cdot \overline{z^{\prime}}  + i|z^{\prime}|^2 \right),
\quad u\in \calU.
\]
All these mappings are holomorphic automorphisms of $\calU$. See \cite[Chapter XII]{Ste93}. Hence the mappings $\sigma_z := \delta_{\bfrho(z)^{-1/2}} \circ h_z$
are holomorphic automorphisms of $\calU$. Simple calculations show that $\sigma_z(z)=\bfi:=(0^{\prime},i)$.

\begin{lemma}\label{lem:rho(sigma_a)}
Suppose $z,u,v\in\calU$, we have 
\begin{align}
\bfrho(\sigma_z(u),\sigma_z(v))&= \bfrho(z)^{-1} \bfrho(u,v)\nonumber,\\
\bfrho(\sigma_z^{-1}(u),\sigma_z^{-1}(v))&= \bfrho(z) \bfrho(u,v)\label{eqn:-rho(sigma_a)}.
\end{align}
\end{lemma}

\begin{proof}
First note that
\[
\bfrho(\delta_t(u),\delta(v))=t^2 \bfrho(u,v).
\]
Also, an easy calculation shows that
\[
\bfrho(h_z(u),h_z(v))=\bfrho(u,v).
\]
Then a combination of the two above equalities gives
\begin{align*}
\bfrho(\sigma_a(z),\sigma_a(w)) &=\bfrho(\delta_{\bfrho(a)^{-1/2}}(h_a(z)),\delta_{\bfrho(a)^{-1/2}}(h_a(w)))\\
&=\bfrho(a)^{-1} \bfrho( h_a(z),h_a(w))\\
&= \bfrho(a)^{-1} \bfrho(z,w),
\end{align*}
which is the first equality. 
The proof of the second one is exactly the same by observing that 
\[
\bfrho(h_z^{-1}(u),h_z^{-1}(v))=\bfrho(u,v)
\]
and 
\[
\sigma_z^{-1}=h_z^{-1}\circ \delta_{\bfrho(z)^{1/2}}.
\]
So we are done.
\end{proof}

We recall from \cite{LS20} that the Bergman metric on $\calU$ is given by 
\[
 \beta(z,w)=\tanh^{-1}\sqrt{1-\frac{\bfrho(z)\bfrho(w)}{|\bfrho(z,w)|^2}},
\]
and the Bergman metric ball at $z$ with radius $r>0$ is denoted by
 \[
 D(z,r) = \{w\in\calU: \beta(z,w)<r\}.
 \]
A sequence $\{a_k\}$ in $\calU$ is called an $r$-lattice in the Bergman metric, or an $r$-lattice for short, 
if it satisfies the following conditions:
\begin{enumerate}
  \item[(i)] $\calU=\bigcup_{k=1}^{\infty} D(a_k,r)$;
  \item[(ii)] $\beta(a_i,a_j)\geq r/2$ for all $i \neq j$.
\end{enumerate}
For any $r>0$, the existence of an $r$-lattice can be verified by the proof of \cite[Theorem 2.23]{Zhu05}. 
Also, associated with a standard maximality argument, we have the following lemma.
\begin{lemma}\label{lem:decomposition}
If $\{a_k\}$ is an $r$-lattice, then it has the following properties:
\begin{enumerate}
  \item[(i)] For any $R>0$, there exists a positive integer $N$ (depending on $r$ and $R$) such that each point in $\calU$ belongs to at most $N$ of the sets $\{D(a_k,R)\}$.
  \item[(ii)] For any $R>0$, there is a finite partition 
$\{a_k\}=\Gamma_1\cup\cdots\cup\Gamma_m$ such that for every $i \in\{1,\cdots,m\}$, the conditions $u,v\in\Gamma_i$ and $u\neq v$ imply that $\beta(u,v)>R$.
\end{enumerate}
\end{lemma}

We need the following result \cite[Lemma 5]{LLHZ18}, which will be used in a few times. 
\begin{lemma}
Let $s,t\in \mathbb{R}$. Then we have
\begin{equation}\label{eqn:keylem}
\int\limits_{\calU} \frac {\bfrho(w)^{t}} {|\bfrho(z,w)|^{s}} dV(w) ~=~
\begin{cases}
\dfrac {C(n,s,t)} {\bfrho(z)^{s-t-n-1}}, &
\text{ if } t>-1 \text{ and } s-t>n+1\\[12pt]
+\infty, &  otherwise
\end{cases}
\end{equation}
for all $z\in \calU$, where
\[
C(n,s,t):=\frac {4 \pi^{n} \Gamma(1+t) \Gamma(s-t-n-1)} {\Gamma^2\left(s/2\right)}.
\]
\end{lemma}

The following results can be found in \cite{LS20}, they serve as basic tools in this paper.

\begin{lemma}
We have
\begin{equation}\label{eqn:elemtryeq1}
 2|\bfrho(z,w)|\geq \max\{\bfrho(z),\bfrho(w)\}
\end{equation}
for any $z, w\in \calU$.
\end{lemma}

\begin{lemma}\label{lem:volBergmanball}
For any $z\in\mathcal{U}$ and $r>0$ we have
\begin{equation}\label{eq:volBergmanball}
|D(z,r)| ~=~ \frac {4\pi^n}{n!} \frac {\tanh^{2n} r } {(1-\tanh^2 r)^{n+1}}\, \bfrho(z)^{n+1}.
\end{equation}
Consequently, the averaging function
\begin{equation}\label{eq:averaging}
\widehat{\mu}_r (z) ~=~ \frac {n!}{4\pi^n} \frac {(1-\tanh^2 r)^{n+1}} {\tanh^{2n} r}
\frac {\mu(D(z,r))} {\bfrho(z)^{n+1}}.
\end{equation}
\end{lemma}


\begin{lemma}\label{blem}
Given $r > 0$, the inequalities
\begin{equation}\label{eqn:eqvltquan}
\frac {1-\tanh (r)}{1+ \tanh (r)} ~\leq~ \frac{|\bfrho(z,u)|}{|\bfrho(z,v)|}
~\leq~ \frac {1+\tanh (r)}{1-\tanh (r)}
\end{equation}
hold for all $z,u,v\in \calU$ with $\beta(u,v)\leq r$.
\end{lemma}

\begin{lemma}\label{lem:similar.Zhu 2.24}
Suppose $r>0$ and $p>0$. Then there exists a positive constant $C$ depending on $r$ such that
\[
|f(z)|^p \leq \frac{C}{\bfrho(z)^{n+1}}\int_{D(z,r)} |f(w)|^p dV(w)
\]
for all $f\in H(\calU)$ and all $z\in\calU$.
\end{lemma}

The next lemma is a key result of \cite{LS20}, we shall use it without any explanation, since it is so natural and important. 
\begin{lemma}
Let $\mu\in\mathcal{M}_+$ and $T_{\mu}$ is bounded on $A^2(\calU)$. Then the equality 
\[
\langle T_{\mu}f,g \rangle = \int_{\calU} f\overline{g} d\mu
\]
holds for all $f,g\in A^2(\calU)$, where $\langle \cdot,\cdot \rangle$ denotes the inner product on $A^2(\calU)$.
\end{lemma}

\section{Averaging functions}

In this section we characterize the $L^p(d\lambda)$-behavior of averaging functions as well as its discretized version. We begin with an observation from \eqref{eq:volBergmanball} and \eqref{eqn:eqvltquan} that there is a positive constant $C$ depending only on $r$ such that
\begin{equation}\label{eq:lambdaBball}
\lambda(D(z,r)) \leq C
\end{equation} 
for all $z\in\calU$.

\begin{lemma}\label{lem:averinlp1}
Suppose $\mu\geq 0$, $r,\delta>0$ and $0<p<\infty$. 
If $\{\widehat{\mu}_{r}(a_k)\}\in l^p$ for some $r$-lattice $\{a_k\}$,
then $\widehat{\mu}_{\delta} \in L^p(\calU,d\lambda)$.
\end{lemma}

\begin{proof}
Assume that $\{a_k\}$ is an $r$-lattice such that $\{\widehat{\mu}_{r}(a_k)\}\in l^p$. Given $z\in\calU$, 
let 
\[
N(z):=\{k:D(a_k,r)\cap D(z,\delta) \neq \emptyset\}.
\]
Since $\{a_k\}$ is an $r$-lattice, we have $D(z,\delta) \subset \cup_{k\in N(z)} D(a_k,r)$.
Thus, 
\[
\mu(D(z,\delta)) \leq \sum_{k\in N(z)} \mu(D(a_k,r)).
\]
Together this with \eqref{eq:averaging} and \eqref{eqn:eqvltquan}, there exists a positive constant $C$ depending on $r$ and $\delta$ such that
\begin{equation}\label{eq:averinlp1}
\widehat{\mu}_{\delta}(z) \leq C \sum_{k\in N(z)} \frac{\bfrho(a_k)^{n+1}}{\bfrho(z)^{n+1}} \widehat{\mu}_r(a_k)
\leq C  \sum_{k\in N(z)} \widehat{\mu}_r(a_k).
\end{equation}

On the other hand, by Lemma \ref{lem:decomposition}, we see that $D(z,\delta)$ meets at most $N$ of the sets $D(a_k,r)$. Therefore, 
\[
\sup_{z\in\calU} |N(z)| \leq N<\infty,
\]
where $|N(z)|$ denotes the number of elements in $N(z)$. This together with \eqref{eq:averinlp1} gives 
\[
\widehat{\mu}_{\delta}(z)^p \leq C^p N^p \sum_{k\in N(z)} \widehat{\mu}_r(a_k)^p
\]
for all $z\in\calU$. Now, integrating both sides of the above against the measure $d\lambda$ and then applying Fubini's theorem, we get that there is another positive constant $C$ such that
\[
\int_{\calU} \widehat{\mu}_{\delta}(z)^p \lambda(z) \leq C \int_{\calU} \sum_{k\in N(z)} \widehat{\mu}_r(a_k)^p d\lambda(z) = C \sum_{k=1}^{\infty} \widehat{\mu}_r (a_k)^p \lambda(Q(k)),
\]
where $Q(k)=\{z\in\calU:D(z,\delta)\cap D(a_k,r) \neq \emptyset \}$. 
Note that $Q(k)\subset D(a_k,r+\delta)$ and by \eqref{eq:lambdaBball} there exists a positive constant $C$ depending on $r+\delta$ such that $\lambda(D(a_k,r+\delta))\leq C$ for all $k\geq 1$. Thus $\lambda(Q(k)) \leq C$ for all $k\geq 1$.
Combining this with the above estimate, we conclude that
\[
\int_{\calU} \widehat{\mu}_{\delta}(z)^p \lambda(z) \leq C \sum_{k=1}^{\infty} \widehat{\mu}_r (a_k)^p.
\]
This completes the proof of the lemma.
\end{proof}

%

\begin{lemma}\label{lem:averinlp2}
Suppose $\mu\geq 0$, $r>0$ and $0<p<\infty$. If $\widehat{\mu}_{2r} \in L^p(\calU,d\lambda)$, then for every $r$-lattice $\{a_k\}$ we have $\{\widehat{\mu}_r(a_k)\}\in l^p$.
\end{lemma}
\begin{proof}
Fix an $r$-lattice $\{a_k\}$. If $z\in D(a_k,r)$, then $D(a_k,r) \subset D(z,2r)$. By \eqref{eq:averaging} and \eqref{eqn:eqvltquan}, there exists a positive constant $C$ depending only on $r$ such that
\[
\widehat{\mu}_r(a_k) \leq C \frac{\mu(D(a_k,r))}{\bfrho(a_k)^{n+1}} \leq C \frac{\mu(D(z,2r))}{\bfrho(z)^{n+1}}\leq C \widehat{\mu}_{2r}(z)
\]
for $z\in D(a_k,r)$. This together with \eqref{eq:lambdaBball} that $\lambda(D(a_k,r)) \leq C$ for $k\geq 1$ gives 
\[
\widehat{\mu}_r(a_k)^p \leq C \int_{D(a_k,r)} \widehat{\mu}_{2r}(z)^p d\lambda(z),\quad k\geq1.
\]
Hence, 
\begin{align*}
\sum_{k=1}^{\infty} \widehat{\mu}_r(a_k)^p &\leq C \sum_{k=1}^{\infty} \int_{D(a_k,r)} \widehat{\mu}_{2r}(z)^p d\lambda(z)\\
&\leq CN \int_{\calU} \widehat{\mu}_{2r}(z)^p d\lambda(z),
\end{align*}
where $N$ is as in Lemma \ref{lem:decomposition}. This completes the proof of the lemma.
\end{proof}

As a consequence of the two lemmas above, we have the following result.

\begin{theorem}\label{thm:main5}
Suppose $\mu\geq 0$, $r,s,\delta>0$ and $0<p<\infty$. Then the following conditions are equivalent:
\begin{enumerate}
  \item[(i)] $\{\widehat{\mu}_r(a_k)\} \in l^p$ for every $r$-lattice $\{a_k\}$.
  \item[(ii)] $\{\widehat{\mu}_s(a_m)\} \in l^p$ for some $s$-lattice $\{a_m\}$.
  \item[(iii)] $\widehat{\mu}_{\delta}\in L^p(\calU,d\lambda)$.
\end{enumerate}
\end{theorem}

\section{The case $p\geq 1$}\label{sec:p>=1}

In this section we prove Theorem \ref{thm:main1} in the case of $p\geq 1$. The left case $0<p<1$ will be discussed in next two sections. Before it, we shall briefly review the notion of Schatten class operators. 

For a positive compact operator $T$ on a a separable Hilbert space $H$, there exists an orthonormal set $\{e_k\}$ in $H$ and a sequence $\{\lambda_k\}$ that decreases to $0$ such that
\[
Tx=\sum_k \lambda_k \langle x,e_k \rangle e_k
\]
for all $x\in H$, where $\langle,\rangle$ denotes the inner product on $H$. For $0<p<\infty$, we say that a positive operator $T$ belongs to the Schatten class $S_p(H)$ if 
\[
\|T\|_p :=\left[ \sum_k \lambda_k^p \right]^{1/p}<\infty.
\]
More generally, given a compact (not necessarily positive) operator $T$ on $H$, we say that $T\in S_p(H)$ if the positive operator $|T|=(T^* T)^{1/2}$ belongs to $S_p(H)$. In this case, we define $\|T\|_p=\| |T| \|_p$.

We shall recall some basic facts about $S_p(H)$, which we need later. Details can be found in \cite{Zhu07}.
\begin{enumerate}
  \item For $T\in S_1(H)$ and an orthonormal basis $\{e_k\}$ for $H$, the sum
  \[
  \mathrm{tr}(T) = \sum_k \langle Te_k,e_k \rangle
  \]
  is absolutely convergent and independent of the choice of $\{e_k\}$. 
  This sum is called the trace of $T$.
  \item For a positive compact operator $T$ on $H$ and $0<p<\infty$, 
  $T\in S_p(H)$ if and only if $T^p\in S_1$ and $\|T\|_p^p=\|T^p\|_1$.
  \item Suppose $T$ is a positive operator on $H$ and $x$ is a unit vector in $H$. 
  If $p\geq 1$, then $\langle T^p x,x \rangle \geq \langle Tx,x \rangle^p$.
  \item For a compact operator $T$ on $H$, if $p\geq 1$, then $T\in S_p(H)$ if and only if 
  \[
  \sup \sum_k |\langle Te_k,e_k \rangle|^p<\infty,
  \]
  where the supremum is taken over all orthonormal set $\{e_k\}$. 
  Moreover, the left side of the above is the same as $\|T\|_p^p$ for $T\in S_p(H)$.
\end{enumerate}

We will take $H=A^2(\calU)$ in our considerations and, in that case, hereafter, we write $S_p=S_p(A^2(\calU))$. 

\begin{lemma}\label{lem:KinL1(mu)}
Let $\mu\in\mathcal{M}_+$. If $K(z,z)\in L^1(\mu)$, then $T_{\mu}$ is compact.
\end{lemma}
\begin{proof}
Note that $K(z,z)=\|K_z\|_2^2$, thus we have
\begin{align*}
\widetilde{\mu}(z) &=\int_{\calU} \left\{\frac{|K_z(w)|}{\|K_z\|_2 \|K_w\|_2}\right\}^2 K(w,w) d\mu(w)\\
&:=\int_{\calU} F(z,w) K(w,w) d\mu(w)
\end{align*}
On one hand, from the proof of \cite[Lemma 2.12]{LS20} we can see that $K_z \|K_z\|_2^{-1}\to 0$ uniformly on every compact subset of $\calU$ as $z\to b\calU\cup\{\infty\}$, where $b\calU=\{z\in \mathbb{C}^n: \bfrho(z)=0\}$ denotes the boundary of $\calU$. Therefore, for every $w\in\calU$, $F(z,w)\to 0$ as $z\to b\calU\cup\{\infty\}$. On the other hand, it follows by \eqref{eqn:elemtryeq1} that $F(z,w)$ is bounded. Hence, the dominated convergence theorem allows us to take the limit inside the integral of the above and deduce that $\widetilde{\mu}$ vanishes on $b\calU\cup\{\infty\}$. This implies by \cite[Theorem 1.2]{LS20} that $T_{\mu}$ is compact.
\end{proof}

Before proving the following lemma, we need to clarify that the well definition of $T_f$ with $f\in L^p(\calU,d\lambda)$ for $1\leq p<\infty$. Let $d\mu=|f|dV$, it suffices to show that 
\begin{equation}\label{eq:T_finSp}
\int_{\calU} \frac{d\mu(z)}{|\bfrho(\bfi,z)|^{\alpha}}<\infty
\end{equation}
for sufficiently large $\alpha$. When $1<p<\infty$, let $p^{\prime}=p/(p-1)$ be the conjugate index of $p$,  by the H\"older's inequality we have 
\begin{align*}
\int_{\calU} \frac{d\mu(z)}{|\bfrho(\bfi,z)|^{\alpha}} &= 
\int_{\calU} \frac{|f(z)| dV(z)}{|\bfrho(\bfi,z)|^{(n+1)/p+\alpha-(n+1)/p}}\\
&\leq \left[\int_{\calU} \frac{|f(z)|^p}{|\bfrho(\bfi,z)|^{n+1}} dV(z)\right]^{1/p}
\left[\int_{\calU} \frac{dV(z)}{|\bfrho(\bfi,z)|^{(\alpha-(n+1)/p)p^{\prime}}}\right]^{1/p^{\prime}}\\
&\leq C \left[\int_{\calU} |f(z)|^p K(z,z) dV(z)\right]^{1/p}
\left[\int_{\calU} \frac{dV(z)}{|\bfrho(\bfi,z)|^{(\alpha-(n+1)/p)p^{\prime}}}\right]^{1/p^{\prime}},
\end{align*}
where the last inequality follows from \eqref{eqn:elemtryeq1}.
By \eqref{eqn:keylem}, the second integral of above is finite if and only if 
$(\alpha-(n+1)/p)p^{\prime}>n+1$, which is equivalent to $\alpha>n+1$. Hence, in the case of $1<p<\infty$, \eqref{eq:T_finSp} holds for $\alpha>n+1$. To prove the remained case $p=1$, by letting $\alpha=n+1$ and using \eqref{eqn:elemtryeq1} again, we obtain
\[
\int_{\calU} \frac{d\mu(z)}{|\bfrho(\bfi,z)|^{\alpha}}=\int_{\calU} \frac{|f(z)|}{|\bfrho(\bfi,z)|^{n+1}} dV(z) 
\leq C \int_{\calU} |f(z)| K(z,z) dV(z),
\]
as desired.

\begin{lemma}\label{lem:T_finSp}
Let $1\leq p<\infty$. If $f\in L^p(\calU,d\lambda)$, then $T_f \in S_p$.
\end{lemma}
\begin{proof}
By the Riesz-Thorin type interpolation theorem for Schatten classes (see \cite[Section 2.2]{Zhu07} for example),
it suffices to show that the map $f\mapsto T_f$ is bounded from $L^1(\calU,d\lambda)$ into $S_1$ (the case $p=\infty$ is trivial). So, assume that $f\in L^1(\calU,d\lambda)$. Then $K(z,z)\in L^1(\mu)$ for $d\mu=|f|dV$.
According to Lemma \ref{lem:KinL1(mu)}, we get that $T_{\mu}$ which is $T_{|f|}$ is compact, and so is $T_f$. To estimate the trace of $T_f$, let $\{e_k\}$ be any orthonormal set in $A^2(\calU)$. Note that 
\[
\langle T_f e_k, e_k \rangle= \int_{\calU} f |e_k|^2 dV
\]
for each $k$. Also, recall that $\sum |e_k(z)|^2 \leq K(z,z)$. Therefore, we have 
\[
\sum_k \left|\langle T_f e_k, e_k \rangle\right| \leq  \int_{\calU} |f| \sum_k |e_k|^2 dV \leq \int_{\calU} |f| d\lambda.
\]
It follows that $T_f\in S_1$ and 
\[
\|T_f\|_1 \leq \int_{\calU} |f| d\lambda.
\]
This completes the proof of the lemma.
\end{proof}

For a bounded linear operator on $A^2(\calU)$, the Berezin transform $\widetilde{T}$ of $T$ is defined by 
\[
\widetilde{T}(z)=\langle Tk_z,k_z \rangle, \quad z\in\calU.
\]
For $\mu\in\mathcal{M}_+$, if $T_{\mu}$ is bounded on $A^2(\calU)$, then $\widetilde{T_{\mu}}=\widetilde{\mu}$.

\begin{lemma}
If $T$ is a positive operator on $A^2(\calU)$, then $T\in S_1$ if and only if $\widetilde{T}\in L^1(\calU,\lambda)$. Moreover, 
\begin{equation}\label{eq:trace}
\mathrm{tr}(T)=\int_{\calU} \widetilde{T}(z) d\lambda(z).
\end{equation}
\end{lemma}
\begin{proof}
The proof is similar to that of \cite[Theorem 6.4]{Zhu07}, so we omit the details here.
\end{proof}

\begin{lemma}\label{lem:TuleqTaver}
Suppose $r>0$ and $\mu\in\mathcal{M}_+$. If $T_{\widehat{\mu}_r}$ is bounded on $A^2(\calU)$, then so is $T_{\mu}$ with $T_{\mu}\leq C_r T_{\widehat{\mu}_r}$ for a constant $C_r>0$.
\end{lemma}
\begin{proof}
Given $f\in A^2(\calU)$, Fubini's theorem gives 
\begin{align*}
\langle T_{\widehat{\mu}_r} f,f \rangle &= \int_{\calU} |f(z)|^2 \widehat{\mu}_r dV(z)\\
&= \int_{\calU} |f(z)|^2 \frac{\mu(D(z,r))}{|D(z,r)|} dV(z)\\
&= \int_{\calU} \frac{|f(z)|^2}{|D(z,r)|} dV(z) \int_{D(z,r)}d\mu(w)\\
&= \int_{\calU} d\mu(w) \int_{D(w,r)} \frac{|f(z)|^2}{|D(z,r)|} dV(z).
\end{align*}
By \eqref{eq:volBergmanball} and \eqref{eqn:eqvltquan}, we can see that $|D(z,r)|$ is comparable to $|D(w,r)|$ for all $w\in\calU$ and $z\in D(w,r)$. This together with Lemma \ref{lem:similar.Zhu 2.24} implies that there exists a positive constant $C$ such that 
\[
|f(w)|^2 \leq C \int_{D(w,r)}  \frac{|f(z)|^2}{|D(z,r)|} dV(z)
\]
for all $w\in\calU$. It follows that 
\[
\langle T_{\mu} f,f \rangle=\int_{\calU} |f(w)|^2 d\mu(w) \leq C \langle T_{\widehat{\mu}_r} f,f \rangle 
\]
for all $f\in A^2(\calU)$, completing the proof of the lemma.
\end{proof}

\begin{theorem}\label{thm:main2}
Suppose $p\geq 1$, $r>0$ and $\mu\in\mathcal{M}_+$. Then the following conditions are equivalent:
\begin{enumerate}
  \item[(i)] $T_{\mu}\in S_p$.
  \item[(ii)] $\widetilde{\mu}\in L^p(\calU,d\lambda)$.
  \item[(iii)] $\widehat{\mu}_r \in L^p(\calU,d\lambda)$.
  \item[(iv)] $\{\widehat{\mu}_r(a_k)\}\in l^p$ for every $r$-lattice $\{a_k\}$.
  \item[(v)] $\{\widehat{\mu}_r(a_k)\}\in l^p$ for some $r$-lattice $\{a_k\}$
\end{enumerate}
\end{theorem}

\begin{proof}
(i) $\Rightarrow$ (ii). Assume $T_{\mu}\in S_p$. Then we have $T_{\mu}^p \in S_1$ and $\widetilde{T_{\mu}^p} \geq (\widetilde{T_{\mu}})^p$ (see \cite[Proposition 1.31]{Zhu07} for example). Combining the fact that $\widetilde{T_{\mu}}=\widetilde{\mu}$ with \eqref{eq:trace}, we have
\[
\int_{\calU} (\widetilde{\mu})^p d\lambda = \int_{\calU} (\widetilde{T_{\mu}})^p d\lambda \leq \int_{\calU} \widetilde{T_{\mu}^p} d\lambda = \mathrm{tr}(T_{\mu}^p) <\infty.
\]

(ii) $\Rightarrow$ (iii). This easily follows by 
\begin{equation}\label{eq:mu_rleqmu}
\widehat{\mu}_r(z)= C \bfrho(z)^{-n-1} \int_{D(z,r)} d\mu(w) \leq C \int_{D(z,r)} |k_z(w)|^2 d\mu(w) \leq C\widetilde{\mu}(z),
\end{equation}
where the expression is due to \eqref{eq:averaging} and \eqref{eqn:eqvltquan}.

(iii) $\Rightarrow$ (i). Assume $\widehat{\mu}_r \in L^p(\calU,d\lambda)$. Thus by Lemma \ref{lem:T_finSp}, $T_{\widehat{\mu}_r} \in S_p$. Then it immediately follows from Lemma \ref{lem:TuleqTaver} that $T_{\mu}\in S_p$. 

That implication (iii) $\Leftrightarrow$ (iv) $\Leftrightarrow$ (v) follows immediately from Theorem \ref{thm:main5}. The proof of the theorem is complete.
\end{proof}

\section{The case $0<p<1$: Part I}\label{sec:p<1P1}

In this section we describe our main result except the integral properties of Berezin transform $\widetilde{\mu}$ in the range of $0<p<1$, where the methods involved are adapted from \cite{Zhu07NY}. 
The key of the section is of characterization of the membership of $T_{\mu}$ in $S_p$ in terms of the averaging function $\widehat{\mu}_r$.
We begin with the following three lemmas, which could be tracked in \cite{Zhu07}.

\begin{lemma}\label{lem:Zhu07Pro1.30}
Suppose $A$ is a bounded surjective operator on $H$ and $T$ is any bounded linear operator on $H$. Then $T\in S_p$ if and only if $A^* T A\in S_p$.
\end{lemma}

\begin{lemma}\label{lem:Zhu07Pro1.29}
Suppose $0<p\leq 2$ and $T$ is a compact operator on $H$. Then
\[
\|T\|_p^p \leq \sum_{i=1}^{\infty} \sum_{j=1}^{\infty} \left| \langle Te_i,e_j \rangle \right|^p
\]
for any orthonormal basis $\{e_k\}$ of $H$.
\end{lemma}

\begin{lemma}\label{lem:Zhu07Cor1.32}
Suppose $T$ is a positive compact operator on $H$ and $\{e_k\}$ is any orthonormal basis of $H$. If $0<p\leq 1$ and 
\[
\sum_{k=1}^{\infty} \langle Te_k,e_k \rangle^p <\infty,
\] 
then $T$ belongs to $S_p$.
\end{lemma}

\begin{lemma}
Suppose $0<p<1$, $r>0$ and $\mu\in\mathcal{M}_+$. If there exists an $r$-lattice $\{a_k\}$ such that $\{\widehat{\mu}_r(a_k)\}\in l^p$, then $T_{\mu}\in S_p$.
\end{lemma}
\begin{proof}
Suppose $\{a_k\}$ is an $r$-lattice such that $\{\widehat{\mu}_r(a_k)\}\in l^p$.
First note from \cite{LS20} that $T_{\mu}$ is compact on $A^2(\calU)$. We want to show that $T_{\mu}$ is in $S_p$. To this end, fix a sufficiently large number $b$ and by \cite[Theorem 2]{CR80} (the atomic decomposition of functions in Bergman spaces on symmetric Siegel domains of type two), we may assume that $A^2(\calU)$ consists exactly of functions of the form 
\[
f(z)=\sum_{k=1}^{\infty} c_k h_k(z),
\] 
where $\{c_k\}\in l^2$, 
\[
h_k(z)= \frac{\bfrho(a_k)^{b-(n+1)/2}}{\bfrho(z,a_k)^b},
\]
and 
\[
\int_{\calU} |f(z)|^2 dV(z) \leq C \sum_{k=1}^{\infty} |c_k|^2
\]
for some positive constant $C$ independent of $\{c_k\}$.

Fix an orthonormal basis $\{e_k\}$ for $A^2(\calU)$ and define an operator $A$ on $A^2(\calU)$ by 
\begin{equation}\label{eq:auxioper}
A \left( \sum_{k=1}^{\infty} c_k e_k \right) = \sum_{k=1}^{\infty} c_k h_k.
\end{equation}
By the statements of above paragraph, we can see that $A$ is a bounded surjective operator on $A^2(\calU)$. Applying Lemma \ref{lem:Zhu07Pro1.30}, the Toeplitz operator $T_{\mu}$ will be in $S_p$ if we can show that the operator $T=A^* T_{\mu} A$ belongs to $S_p$. To show that $T\in S_p$, according to Lemma \ref{lem:Zhu07Cor1.32}, we just need to verify that 
\[
M= \sum_{k=1}^{\infty} \langle Te_k,e_k \rangle^p <\infty.
\]

First we note that 
\[
\langle Te_k,e_k \rangle= \langle  T_{\mu} h_k, h_k \rangle = \int_{\calU} |h_k(z)|^2 d\mu(z) \leq \sum_{j=1}^{\infty} \int_{D(a_j,r)} |h_k(z)|^2 d\mu(z).
\]
By \eqref{eqn:eqvltquan} and \eqref{eq:averaging}, there is a positive constant $C$ such that 
\[
\langle Te_k,e_k \rangle \leq C \sum_{j=1}^{\infty} |h_k(a_j)|^2 \bfrho(a_j)^{n+1} \widehat{\mu}_r(a_j).
\]
Since $0<p<1$, an application of H\"oder's inequality gives 
\[
\langle Te_k,e_k \rangle^p \leq C \sum_{j=1}^{\infty} |h_k(a_j)|^{2p}\bfrho(a_j)^{p(n+1)} \widehat{\mu}_r(a_j)^p.
\]
Thus by Fubini's theorem, we obtain
\[
M \leq C \sum_{j=1}^{\infty} \widehat{\mu}_r(a_j)^p  \bfrho(a_j)^{p(n+1)} \sum_{k=1}^{\infty} |h_k(a_j)|^{2p}.
\]

For each $j$ we consider the sum
\[
M_j = \sum_{k=1}^{\infty} |h_k(a_j)|^{2p} = \sum_{k=1}^{\infty} \frac{\bfrho(a_k)^{p(2b-n-1)}}{|\bfrho(a_j,a_k)|^{2pb}}.
\]
By Lemma \ref{lem:similar.Zhu 2.24}, there exists a positive constant $C$ such that 
\[
\frac{1}{|\bfrho(a_j,a_k)|^{2pb}} \leq  \frac{C}{\bfrho(a_k)^{n+1}} \int_{D(a_k,r)} \frac{dV(z)}{|\bfrho(a_j,z)|^{2pb}}
\]
for all $j$ and all $k$. Since $\bfrho(a_k)$ is comparable to $\bfrho(z)$ for $z\in D(a_k,r)$, we have
\begin{align*}
M_j &\leq C \sum_{k=1}^{\infty} \int_{D(a_k,r)} \frac{\bfrho(z)^{p(2b-n-1)-n-1}}{|\bfrho(a_j,z)|^{2pb}} dV(z)\\
&\leq CN \int_{\calU} \frac{\bfrho(z)^{p(2b-n-1)-n-1}}{|\bfrho(a_j,z)|^{2pb}} dV(z),
\end{align*}
where $N$ is as in Lemma \ref{lem:decomposition}. We can assume that $b$ is large enough so that $p(2b-n-1)>n$, then applying \eqref{eqn:keylem}, there exists a positive constant $C$ such that 
\[
M_j \leq C \bfrho(a_j)^{-p(n+1)}
\]
for all $j$. Hence, it follows that 
\[
M\leq C \sum_{j=1}^{\infty} \widehat{\mu}_r(a_j)^p <\infty.
\]
This completes the proof of the lemma.
\end{proof}

\begin{lemma}
Suppose $0<p<1$, $r>0$ and $\mu\in\mathcal{M}_+$. If $T_{\mu}\in S_p$ and $\{a_k\}$ is an $r$-lattice, then $\{\widehat{\mu}_r(a_k)\}\in l^p$.
\end{lemma}

\begin{proof}
Fix a sufficiently large positive number $R$. Lemma \ref{lem:decomposition} tells us that
there is a decomposition of $\{a_k\}$ into $m$ subsequences $\{\Gamma_i\}$ such that for every pair $u,v\in\Gamma_i$ with $u\neq v$, $\beta(u,v)>R$. Let $\{\zeta_j\}$ be some $\Gamma_i$ and define a measure $\nu$ as follows:
\[
d\nu(z)=\sum_{k=1}^{\infty} \chi_k(z) d\mu(z),
\]
where $\chi_k$ is the characteristic function of $D(\zeta_k,r)$. Assume that $R>2r$, then the Bergman metric balls $\{D(\zeta_k,r)\}$ are disjoint.
Also, note that $0\leq\nu\leq\mu$, we have $\nu\in\mathcal{M}_+$ and $T_{\nu}\in S_p$ with 
$\|T_{\nu}\|_p \leq \|T_{\mu}\|_p$. 

Fix an orthonormal basis $\{e_k\}$ for $A^2(\calU)$. Similar to \eqref{eq:auxioper}, we define an auxiliary bounded operator 
\[
A \left( \sum_{k=1}^{\infty} c_k e_k \right) = \sum_{k=1}^{\infty} c_k h_k,
\]
where 
\[
h_k(z)= \frac{\bfrho(\zeta_k)^{b-(n+1)/2}}{\bfrho(z,\zeta_k)^b}
\]
and $b$ is sufficiently large.

Put $T=A^* T_{\nu}  A$. Since $A$ is bounded and $T_{\nu}\in S_p$, we have $T\in S_p$ with 
$\|T\|_p\leq \|A\|^2 \|T_{\nu}\|_p$. Hence, there exists a positive constant $C$ such that 
\begin{equation}\label{eq:tria_T}
\|T\|_p^p\leq C\|T_{\mu}\|_p^p.
\end{equation}

We split the operator $T$ as $T=D+E$, where $D$ is the diagonal operator defined by
\[
Df = \sum_{k=1}^{\infty} \langle Te_k,e_k \rangle \langle f,e_k \rangle e_k, \quad f\in A^2(\calU),
\]
and $E=T-D$. 

Note that $D$ is compact and positive, we have
\begin{align*}
\|D\|_p^p &= \sum_{k=1}^{\infty} \langle Te_k,e_k \rangle^p 
= \sum_{k=1}^{\infty} \langle T_{\nu}h_k,h_k \rangle^p \\
&=\sum_{k=1}^{\infty} \left[ \int_{\calU} |h_k(z)|^2 d\nu(z) \right]^p
\geq \sum_{k=1}^{\infty} \left[ \int_{D(\zeta_k,r)} |h_k(z)|^2 d\nu(z) \right]^p \\
&\geq C \sum_{k=1}^{\infty} \widehat{\nu}_r(\zeta_k)^p,
\end{align*}
where the last inequality follows by \eqref{eq:averaging} and \eqref{eqn:eqvltquan}. Since $\nu=\mu$ on each $D(\zeta_k,r)$, we obtain
\begin{equation}\label{eq:lownormD}
\|D\|_p^p \geq C_1 \sum_{k=1}^{\infty} \widehat{\mu}_r(\zeta_k)^p.
\end{equation}

On the other hand, by Lemma \ref{lem:Zhu07Pro1.29} we have 
\begin{align*}
\|E\|_p^p & \leq \sum_{i=1}^{\infty} \sum_{j=1}^{\infty} \left| \langle Ee_i,e_j \rangle \right|^p 
=  \sum_{j\neq k} \left| \langle T_{\nu}h_j,h_k \rangle \right|^p \\
& = \sum_{j\neq k} \left| \int_{\calU} h_j(z) \overline{h_k(z)} d\nu(z) \right|^p \\
& \leq \sum_{j\neq k} \left[ \sum_{i=1}^{\infty}  \int_{D(\zeta_i,r)} |h_j(z) h_k(z)| d\mu(z) \right]^p.
\end{align*}
Again by \eqref{eq:averaging} and \eqref{eqn:eqvltquan}, there exists a positive $C$ such that 
\[
\int_{D(\zeta_i,r)} |h_j(z) h_k(z)| d\mu(z) 
\leq C \bfrho(\zeta_i)^{n+1} |h_j(\zeta_i) h_k(\zeta_i)| \widehat{\mu}_r (\zeta_i)
\]
for all $i$.
Since $0<p<1$, an application of H\"older's inequality gives
\[
\|E\|_p^p \leq C \sum_{j\neq k}
\sum_{i=1}^{\infty} \bfrho(\zeta_i)^{p(n+1)} |h_j(\zeta_i) h_k(\zeta_i)|^p \widehat{\mu}_r (\zeta_i)^p.
\]
Then using Fubini's theorem, we obtain
\[
\|E\|_p^p \leq C \sum_{i=1}^{\infty} \bfrho(\zeta_i)^{p(n+1)} \widehat{\mu}_r (\zeta_i)^p I_i,
\]
where 
\[
I_i = \sum_{j\neq k} |h_j(\zeta_i) h_k(\zeta_i)|^p.
\]
Since every $|h_j(\zeta_i)|^p$ is comparable to 
\[
\int_{D(\zeta_j,r)} \frac{\bfrho(z)^{pb-p(n+1)/2}}{|\bfrho(\zeta_i,z)|^{pb}} d\lambda(z)
\]
by \eqref{eq:volBergmanball} and \eqref{eqn:eqvltquan}, and since 
$\Omega=\bigcup_{j\neq k} D(\zeta_j,r) \times D(\zeta_k,r)$
is a disjoint union, we can find a positive constant $C$ such that
\[
I_{i} \leq C \iint_{\Omega} \frac{[\bfrho(z) \bfrho(w)]^{pb-p(n+1)/2}}{[|\bfrho(\zeta_i,z)| |\bfrho(\zeta_i,w)]^{pb}} d\lambda(z)d\lambda(w).
\]
By assumption, we have 
\[
\Omega \subset G_R = \{(z,w)\in \calU \times \calU : \beta(z,w)\geq R-2r\}.
\]
Hence, 
\[
I_{i} \leq C \iint_{G_R} \frac{[\bfrho(z) \bfrho(w)]^{pb-p(n+1)/2}}{[|\bfrho(\zeta_i,z)| |\bfrho(\zeta_i,w)]^{pb}} d\lambda(z)d\lambda(w).
\]
Making the change of $z=\sigma_{\zeta_i}^{-1}(u)$ and $w=\sigma_{\zeta_i}^{-1}(v)$, by \eqref{eqn:-rho(sigma_a)} we obtain  
\[
I_i \leq C  \bfrho(\zeta_i)^{-p(n+1)} \iint_{G_R} \frac{[\bfrho(u) \bfrho(v)]^{pb-p(n+1)/2}}{[|\bfrho(\bfi,u)| |\bfrho(\bfi,v)]^{pb}} d\lambda(u)d\lambda(v).
\]
Therefore, there is a positive constant $C_2$ such that 
\begin{equation}\label{eq:upnormE}
\|E\|_p^p \leq C_2 C_R \sum_{i=1}^{\infty} \widehat{\mu}_r (\zeta_i)^p,
\end{equation}
where 
\[
C_R = \iint_{G_R} \frac{[\bfrho(u) \bfrho(v)]^{pb-p(n+1)/2-n-1}}{[|\bfrho(\bfi,u)| |\bfrho(\bfi,v)]^{pb}} dV(u)dV(v).
\]
We can assume that $b$ is large enough so that $pb>p(n+1)/2+n$, thus it follows by \eqref{eqn:keylem} that 
\[
\int_{\calU} \frac{\bfrho(u)^{pb-p(n+1)/2-n-1}}{|\bfrho(\bfi,u)|^{pb}} dV(u)<\infty.
\]
Consequently, it follows that $C_R \to 0$ as $R\to\infty$.

Finally, by the triangle inequality, we know that
\[
\|T\|_p^p\geq \|D\|_p^p - \|T\|_p^p.
\]
In view of \eqref{eq:lownormD} and \eqref{eq:upnormE}, we have
\[
\|T\|_p^p\geq (C_1-C_2 C_R) \sum_{i=1}^{\infty} \widehat{\mu}_r (\zeta_i)^p.
\]
It is clear that $C_1$ and $C_2$ are independent of $R$. We can chose $R$ large enough so that $C_1-C_2 C_R>0$. Together with \eqref{eq:tria_T}, it follows that there exists a positive constant $C$ such that
\[
\sum_{i=1}^{\infty} \widehat{\mu}_r (\zeta_i)^p \leq C \|T_{\mu}\|_p^p.
\] 
Since this holds for each one of the $m$ subsequences of $\{a_k\}$, it follows that 
\[
\sum_{k=1}^{\infty} \widehat{\mu}_r (a_k)^p \leq C m \|T_{\mu}\|_p^p.
\]
This completes the proof of the lemma.
\end{proof}

As a consequence of the two lemmas above, we see that $T_{\mu}\in S_p$ if and only if 
$\{\widehat{\mu}_r(a_k)\} \in l^p$ for some (every) $r$-lattice $\{a_k\}$. Associating with Theorem \ref{thm:main5}, we conclude the main result of the section. 

\begin{theorem}\label{thm:main3}
Suppose $0<p<1$, $r>0$ and $\mu\in\mathcal{M}_+$. Then the following conditions are equivalent:
\begin{enumerate}
  \item[(i)] $T_{\mu}\in S_p$.
  \item[(ii)] $\widehat{\mu}_r \in L^p(\calU,d\lambda)$.
  \item[(iii)] $\{\widehat{\mu}_r(a_k)\}\in l^p$ for every $r$-lattice $\{a_k\}$.
  \item[(iii)] $\{\widehat{\mu}_r(a_k)\}\in l^p$ for some $r$-lattice $\{a_k\}$.
\end{enumerate}
\end{theorem}

\section{The case $0<p<1$: Part II}\label{sec:p<1P2}

In this section we characterize the membership of $T_{\mu}$ in $S_p$ by integral properties of the Berezin transform $\widetilde{\mu}$ in the range of $0<p<1$. It turns out that this can not be done for the full range. We begin the section with showing the obstruction. 

If $\mu$ is any positive Borel measure on $\calU$, a use of \eqref{eqn:eqvltquan} shows that 
\begin{align*}
\widetilde{\mu}(z) &= \frac{n!}{4\pi^n} \int_{\calU} \frac{\bfrho(z)^{n+1}}{|\bfrho(z,w)|^{2(n+1)}} d\mu(w)\\
&\geq \frac{n!}{4\pi^n} \int_{D(\bfi,r)} \frac{\bfrho(z)^{n+1}}{|\bfrho(z,w)|^{2(n+1)}} d\mu(w)\\
&\geq C \mu(D(\bfi,r)) \frac{\bfrho(z)^{n+1}}{|\bfrho(z,\bfi)|^{2(n+1)}}\\
&\geq C_1 \frac{\bfrho(z)^{n+1}}{|\bfrho(z,\bfi)|^{2(n+1)}},
\end{align*}
where $C$ and $C_1$ are positive constants independent of $z$. Thus, an application of \eqref{eqn:keylem} implies that 
\[
\int_{\calU} \widetilde{\mu}(z)^p d\lambda(z) \geq C_1^p  \int_{\calU} \frac{\bfrho(z)^{p(n+1)-n-1}}{|\bfrho(z,\bfi)|^{2p(n+1)}} dV(z)=\infty
\]
whenever $p(n+1)\leq n$. Therefore, in the range $0<p\leq n/(n+1)$, it is not possible to characterize the membership of $T_{\mu}$ in $S_p$ in terms of the $L^p(\calU,d\lambda)$-properties of $\widetilde{\mu}$.
The following result shows that this is the only obstruction.

\begin{theorem}\label{thm:main4}
Suppose $\mu\in\mathcal{M}_+$ and 
\[
n/(n+1)<p<1.
\]
Then $T_{\mu} \in S_p$ if and only if $\widetilde{\mu}\in L^p(\calU,d\lambda)$.
\end{theorem}

\begin{proof}
In view of \eqref{eq:mu_rleqmu}, the condition $\widetilde{\mu}\in L^p(\calU,d\lambda)$ implies that
$\widehat{\mu}_r \in L^p(\calU,d\lambda)$, which, by Theorem \ref{thm:main3}, 
implies that $T_{\mu} \in S_p$.

Next we assume that $T_{\mu} \in S_p$. Given an $r$-lattice $\{a_k\}$, according to Theorem \ref{thm:main3}, we have $\{\widehat{\mu}_r(a_k)\}\in l^p$. 
To show that $\widetilde{\mu}\in L^p(\calU,d\lambda)$, it suffices to prove that the $L^p(\calU,d\lambda)$-norm of $\widetilde{\mu}$ is dominated by a constant multiple of the $l^p$-norm of $\{\widehat{\mu}_r(a_k)\}$.
By \eqref{eqn:eqvltquan} and \eqref{eq:averaging}, there is positive constant $C$ such that
\begin{align*}
\widetilde{\mu}(z) &= \frac{n!}{4\pi^n} \int_{\calU} \frac{\bfrho(z)^{n+1}}{|\bfrho(z,w)|^{2(n+1)}} d\mu(w)\\
&\leq \frac{n!}{4\pi^n} \sum_{k=1}^{\infty} \int_{D(a_k,r)} \frac{\bfrho(z)^{n+1}}{|\bfrho(z,w)|^{2(n+1)}} d\mu(w)\\
&\leq C \sum_{k=1}^{\infty} \frac{\bfrho(z)^{n+1}}{|\bfrho(z,a_k)|^{2(n+1)}} \mu(D(a_k,r))\\
&\leq C \sum_{k=1}^{\infty} \frac{\bfrho(z)^{n+1}}{|\bfrho(z,a_k)|^{2(n+1)}} \bfrho(a_k)^{n+1} \widehat{\mu}_r(a_k).
\end{align*}
An application of H\"older's inequality leads to 
\[
\int_{\calU} \widetilde{\mu}(z)^p d\lambda(z)
\leq C \sum_{k=1}^{\infty} \bfrho(a_k)^{p(n+1)} \widehat{\mu}_r(a_k)^p
\int_{\calU} \frac{\bfrho(z)^{p(n+1)-n-1}}{|\bfrho(z,a_k)|^{2p(n+1)}} dV(z).
\]
Applying \eqref{eqn:keylem}, the integrability of right hand side of above inequality is guaranteed by the assumption that $p(n+1)>n$. Therefore, it follows that there is another positive constant $C$ such that 
\[
\int_{\calU} \widetilde{\mu}(z)^p d\lambda(z) \leq C \sum_{k=1}^{\infty} \widehat{\mu}_r(a_k)^p.
\]
This completes the proof of the theorem.
\end{proof}

Now, a combination of Theorem \ref{thm:main2}, Theorem \ref{thm:main3} and Theorem \ref{thm:main4} gives our main result which was stated as Theorem \ref{thm:main1} in the introduction.

\begin{remark}
At the end of the paper, it should be pointed out that the methods involved in this paper are also applicable to the weighted cases.
\end{remark}


\end{document}